\documentclass[12pt]{amsart}
\usepackage[a4paper,margin=1in,footskip=0.25in]{geometry}
\usepackage{amsmath, amssymb, amsthm, amsfonts}
\usepackage{tikz}
\usetikzlibrary{arrows, babel}
\usepackage{tikz-cd}
\usepackage[shortlabels]{enumitem}
\usepackage{hyperref}
\usepackage[capitalise,noabbrev]{cleveref}

\newtheorem{thm}{Theorem}[section]
\newtheorem{cor}[thm]{Corollary}
\newtheorem{lem}[thm]{Lemma}
\newtheorem{prop}[thm]{Proposition}
\theoremstyle{definition}
\newtheorem{defn}[thm]{Definition}
\theoremstyle{remark}
\newtheorem{rmk}[thm]{Remark}
\newtheorem{exm}[thm]{Example}

\def\RR{\ensuremath{\mathbb{R}}}
\def\QQ{\ensuremath{\mathbb{Q}}}
\def\ZZ{\ensuremath{\mathbb{Z}}}

\newcommand{\ModelCategoryDefinition}[3]{
\begin{description}
\item[$\overset\sim\to$] #1
\item[$\hookrightarrow$] #2
\item[$\twoheadrightarrow$] #3
\end{description}	
}

\newcommand{\Acy}{\mathrm A}
\newcommand{\PP}{\mathrm P}
\newcommand{\LL}{\mathrm L}
\newcommand{\id}{\mathrm{id}}

\newcommand{\Fun}{\operatorname{Fun}}

\newcommand{\Hom}{\operatorname{Hom}}
\newcommand{\hocolim}{\operatorname{hocolim}}

\newcommand{\sSet}{\operatorname{sSet}}
\newcommand{\Arr}{\operatorname{Arr}}
\newcommand{\map}{\operatorname{map}}

\newcommand{\chara}{\operatorname{char}}

\title{Relative plus constructions}

\author{Guille Carri\'on Santiago}

\address{Departament de Matemàtiques, Universitat Aut\`onoma de Barcelona, Barcelona, Spain}
\email{g.carrionsantiago@gmail.com}

\author{J\'er\^ome Scherer}
\address{Mathematics, Ecole Polytechnique F\'ed\'erale de Lausanne, EPFL, Switzerland}
\email{jerome.scherer@epfl.ch}



\begin{document}
\begin{abstract}
Let $h$ be a connective homology theory. We construct a functorial relative plus construction as a Bousfield localization functor in the category of maps of spaces. It allows us to associate to a pair $(X, H)$, consisting of a connected space $X$ and an $h$-perfect normal subgroup $H$ of the fundamental group $\pi_1(X)$, an $h$-acyclic map $X \rightarrow X^{+h}_H$ inducing the quotient by $H$ on the fundamental group. We show that this map is terminal among the $h$-acyclic maps that kill a subgroup of $H$. When $h$ is an ordinary homology theory with coefficients in a commutative ring with unit $R$, this provides a functorial and well-defined counterpart to a construction by cell attachment introduced by Broto, Levi, and Oliver in the spirit of Quillen's plus construction. We also clarify the necessity to use a strongly $R$-perfect group $H$ in characteristic zero.
\end{abstract}
\maketitle

\section{Introduction}

Since Quillen's introduction of the plus-construction in \cite{Quillen1971} as a mean to study the algebraic K-theory groups of a ring, the number of improvements, generalizations, variants, or direct applications has been growing year after year.
Originally Quillen's construction is done by cell attachments. It does not only make sense for $B\mathrm{GL}(R)$, where it gives rise to the K-theory groups $K_n R = \pi_n B\mathrm{GL}(R)^+$ for $n \geq 1$, but for any connected space~$X$. To kill a perfect subgroup of the fundamental group, it suffices to attach cells of dimension $2$; but to do it without changing the homology, one has to attach cells of dimension $3$ to get it right.

One obvious advantage of this construction is that it is economical, the number of cells and their dimension is small, but one drawback is the lack of functoriality. The use of Bousfield localization fixes this problem and relates it to homological localization, as observed by Casacuberta in \cite[Section~6]{MR1290581} and Bousfield in \cite[Subsection~4.3]{MR1481817}. The desire to have a localization functor with respect to ordinary homology with coefficients in a ring $R$ dates back at least to Sullivan, \cite{MR442930}, Hilton-Mislin-Roitberg, \cite{MR0478146}, Adams, \cite{MR0420607}, and Bousfield-Kan, \cite{MR0365573}. The most elegant way to solve this is not the most effective sizewise as it will enjoy a universal property. Bousfield realized, \cite{MR380779}, that inverting formally all homology equivalences between CW-complexes with countably many cells, one actually inverts all homology equivalences. As they form a set, one can form a single map $f$ by taking the wedge of all of them and consider the left Bousfield localization functor $\LL_f = \LL_{HR}$. 

This is a homology localization functor for ordinary homology with coefficients in $R$ and we are interested in the associated nullification functor $\PP_A = \LL_{A \rightarrow *}$ where $A$ is the homotopy cofiber of $f$. 
The space $A$ is acyclic, $\widetilde H_*(A; R) = 0$, being the homotopy cofiber of a homology equivalence. It is a \emph{universal} acyclic space in the sense that $\PP_A X$ is contractible for all acyclic spaces $X$. Therefore $\PP_A$ deserves the name of plus-construction for the homology theory $H_*(-; R)$, we write $\PP_A X = X^{+ R}$.
In principle, one has thus to perform a nullification with respect to a large universal acyclic space. Such a space is not uniquely determined and particularly nice models have been constructed by Berrick and Casacuberta in \cite{Berrick1999} for integral homology. For other coefficients, we recommend Mislin and Peschke's inspiring work \cite{Mislin2001}, see also Tai's \cite{Tai1998}.
Early occurrences of such constructions were called ``partial $R$-completions'' in \cite[VII.6.2]{MR0365573}. 
Functorial constructions are necessary in many applications; let us cite, for example, fiberwise plus constructions as presented in \cite[Section~1.F]{Farjoun1996}.

While a lot of work has been done on (absolute) plus-constructions, \emph{relative} plus-constructions that were present in Quillen's original work have been much less studied, maybe because the functoriality issues are more delicate, even to state them well. By relative, we mean that one does not wish to use all maps from a universal acyclic space, but only those that are necessary to kill a fixed subgroup of the fundamental group. In the recent article \cite{Broto2021}, Broto, Levi, and Oliver needed such a relative plus construction for homology with coefficients in a ring $R$ different from~$\mathbb Z$. A typical application of their techniques is to apply this relative plus construction for mod $p$ homology on $BG$ so as to kill the subgroup $O^p(G)$ of the fundamental group $G$. They find then a chain complex whose homology is the mod $p$ loop space homology $H_*(\Omega(BG)^{\wedge}_p; \mathbb F_p)$, recovering Benson's result in \cite{MR2465835}. The construction in \cite{Broto2021} is done in the spirit of Quillen's original construction, attaching cells in dimension $2$ and $3$ so as to kill a fixed subgroup $H \unlhd \pi_1(X)$ without changing the homology with coefficients in $R$. They prove that such a relative plus construction $X \rightarrow X^{+ R}_H$ exists if and only if $H$ is $R$-perfect when the characteristic of $R$ is different from zero, or if $H$ is \emph{strongly $R$-perfect} when $\chara(R)=0$.


In the present article, we provide complete answers to the following issues. The first one relates to the strongly perfect subgroups and clarifies the reason why the necessary condition for the existence of a partial plus construction seems to be different for different rings. We recall Bousfield's observation about the acyclicity for ordinary homology $H_*(-; R)$ with coefficients in an arbitrary ring $R$, \cite{Bousfield1974}. There is always a ``nice'' abelian group $S$ that we call $h$-characteristic
such that $X^{+R}$ and $X^{+S}$ coincide. 
Relative plus-constructions exist for \emph{all} $S$-perfect subgroups of the fundamental group. 
In fact, we prove in \cref{prop:perfectkernel} that for a ring $R$ of characteristic zero strongly $R$-perfect groups are exactly the $S$-perfect ones. 

Our second contribution, without surprise, relates to the lack of functoriality of the cellular construction. We observe in \cref{exm:BLO-not-well-defined} that the Broto-Levi-Oliver construction is not even well-defined up to homotopy. We propose here a functorial construction in the category of arrows (maps of spaces) allowing us to choose a subgroup $H \unlhd \pi_1(X)$. This also fixes the question of uniqueness. More precisely, we represent a pair $(X, H)$ consisting of a pointed connected space $X$ and a normal subgroup $H$ of the fundamental group by the associated covering space $\widetilde X \rightarrow X$. We choose a universal $HR$-acyclic space $A$ so that the associated nullification functor $\PP_{A}$ is a functorial (absolute) plus-construction: $\PP_{A} X = X^{+ R}$. We move now to the presheaf category of arrows, or morphisms of spaces, where Bousfield localization techniques are available. We have, in particular, a nullification functor $\PP_{\id_{A}}$ associated to the object $\id_A\colon A \rightarrow A$. 
We describe the nullification with respect to any identity map in terms of well-understood spaces.

\medskip

\noindent
{\bf \cref{prop:nullification_respect_identity}.}
\emph{
Let $W$ be a pointed space, and $f\colon  X_0\to X_1$ be a map between pointed connected spaces. Then $\PP_{\id_W}f\colon  Y_0 \to Y_1$, the $\id_{W}$-nullification of $f$, is the homotopy pushout of $f$ along the $W$-nullification $X_0 \rightarrow \PP_W X_0 = Y_0$, i.e., the following diagram is a homotopy pushout:}
\[
\begin{tikzcd}
X_0 \arrow[d, "f"] \arrow[r] & \PP_W X_0 \arrow[d, "\PP_{\id_W}(f)"] \\
X_1 \arrow[r] & Y_1. 
\end{tikzcd}
\]

\medskip

When $W$ is a universal acyclic space $A$, this gives a new way to define relative plus constructions; see \cref{def:universal_plus_construction}. We prove that the map $X \rightarrow X^{+R}_H$ provided by the bottom arrow in the homotopy pushout $p\to\PP_{\id_A}(p)$ is a homology equivalence with coefficients in $R$ inducing the quotient by $H$ on the fundamental group, see \cref{thm:elementaryproperties}. We also show in \cref{thm:universa_plus_construction_local_coeff} that our functorial relative plus-construction induces a homology isomorphism for arbitrary twisted coefficients, which completes the list of properties a relative plus construction should enjoy following Broto, Levi, and Oliver's \cref{def:BLO+const}.

%
%

Given any connective homology theory $h$, we finally prove in \cref{sec:universal} that our relative plus construction $X \rightarrow X^{+h}_H$, which kills the subgroup $H\unlhd \pi_1(X)$, is terminal among all $h$-acyclic maps that kill a subgroup of $H$ on $\pi_1 (X)$, thereby providing a universal property characterizing our construction among all possible choices.

\medskip

\noindent
{\bf \cref{thm:universal property}.}
\emph{Let $h$ be a connective homology theory and $S$ be the associated $h$-character\-istic group. Let $X$ be a connected space and $H\unlhd \pi_1(X)$ an $S$-perfect subgroup of $\pi_1(X)$. The relative plus construction $q\colon X\to X^{+h}_H$ is terminal among all $h$-acyclic maps that kill a subgroup of $H$ on $\pi_1 (X)$.
}

\medskip


\noindent
{\bf Acknowledgments.} This project started during the first author's stay at EPFL in the fall semester of 2022. He would like to thank the Mathematics Institute in Lausanne for its hospitality and the Spanish Agencia Estatal de Investigación for making this possible thanks to the project PID2020-116481GB-I00 and Spanish Ministerio de Ciencia e Innovación grant BES-2017-079851. It is a pleasure to thank Victor Torres for helpful discussions during that semester, Nat\`alia Castellana for many valuable suggestions and useful comments, and the referee without whom we would have forgotten to write the essential \cref{sec:universal}.

\section{Preliminaries}
This section summarizes the results we need to construct a nullification functor that coincides with a relative plus construction. In~\ref{sec:model_cat_structure}, we recall some properties of the simplicial model structure on the category of arrows of spaces, and in~\ref{sec:nul_funct}, we describe the nullification functor in the category of arrows.

\subsection{Model category structure}\label{sec:model_cat_structure}
Let $\sSet_*$ be the category of pointed spaces. The \emph{arrow category} of pointed spaces, which we denote by $\Arr(\sSet_*)$, is the category whose objects are morphisms of pointed spaces and given two objects $f\colon X_0\to X_1$ and \mbox{$g\colon Y_0\to Y_1$}, a morphism between them $\alpha\colon f\to g$ is a pair of morphisms of pointed spaces, \mbox{$\alpha_0\colon X_0\to Y_0$} and $\alpha_1\colon X_1\to Y_1$ such that the following diagram commutes:
\[
\begin{tikzcd}
X_0 \arrow[d, "f"] \arrow[r, "\alpha_0"] & Y_0 \arrow[d, "g"] \\
X_1 \arrow[r, "\alpha_1"]                & Y_1.
\end{tikzcd}
\]

Since the arrow category of pointed spaces, $\Arr(\sSet_*)$, can be identified with the category of functors $\Fun(0\to 1, \sSet_*)$, following Hirschhorn \cite[Section 9.1, Section 11.7]{Hirschhorn2003}, $\Arr(\sSet_*)$ admits a simplicial model category structure induced by the Quillen model category structure in $\sSet_*$, see \cite[7.10.13]{Hirschhorn2003}, or \cite[Section 6]{Balchin2021}.

\begin{thm}\label{thm:model_category_Arr(sSet)}
There is a simplicial model category structure in the category of arrows $\Arr(\sSet_*)$ in which a map $\alpha\colon (X_0\overset f\to X_1) \to (Y_0\overset g\to Y_1)$ is defined to be a:
\ModelCategoryDefinition
{weak equivalence if both $\alpha_0$, $\alpha_1$ are weak equivalences in $\sSet_*$;
}
{a cofibration if $\alpha_0\colon X_0\to Y_0$ and $\alpha_1\colon X_1\to Y_1$ are injective; and}
{a fibration if $\alpha_1\colon X_1\to Y_1$ and the natural map $X_0\to X_1\times_{Y_1}Y_0$ are unpointed Kan fibrations, this is, both maps have he right lifting property with respect to the map $\Lambda[n,k]\to \Delta[n]$ for every $n>0$ and every $0\le k\le n$.}
The simplicial structure, denoted by $\otimes$, is defined as follows: Given an object $f\colon X_0\to X_1$ and a space $K$;
\begin{enumerate}
\item the object $f\otimes K$ is defined by $(f\wedge \id_{K_+})\colon X_0\wedge K_+\to X_1\wedge K_+$; and
\item the object $f^K$ is the induced map on pointed mapping spaces $f_*\colon \map_*(K_+,X_0)\to \map_*(K_+,X_1)$.
\end{enumerate}
\end{thm}

With this simplicial structure, given two arrows of pointed spaces $f\colon X_0\to X_1$, \mbox{$g\colon Y_0\to Y_1$}, the \emph{mapping space} from $f$ to $g$ is the (unpointed) space that in degree $n$ is
\[
\map(f,g)_n:=\Hom_{\Arr(\sSet_*)}(f\wedge \id_{\Delta[n]_+},g)
\]
with face and degeneracy maps induced by the standard maps between $\Delta[n]$'s.  Another useful way to understand mapping spaces is to describe them as being built from usual mapping spaces of simplicial sets. The description of morphisms as commuting squares tells us that $\map(f, g)$ is given as a pullback in
\[
\begin{tikzcd}
{\map(f,g)} \arrow[d] \arrow[r]  & {\map(X_1,Y_1)} \arrow[d, "f^*"] \\
{\map(X_0,Y_0)} \arrow[r, "g_*"] & {\map(X_0,Y_1)}.                 
\end{tikzcd}
\]
\begin{cor}
In this model category structure, an object $f\colon X_0\to X_1$ is:
\begin{itemize}
    \item cofibrant if $X_0$ and $X_1$ are cofibrant pointed simplicial sets; and
    \item fibrant if and only if $f$ is a Kan fibration with $X_1$ a Kan complex.
\end{itemize}
\end{cor}

\subsection{Nullification in the arrow category}\label{sec:nul_funct}

The way we produce a functorial relative plus construction is by repackaging the information given by a space and a normal subgroup of its fundamental group as a covering of that space. We need, therefore, some notation and a few results about nullification functors in the arrow category. The simplicial model structure described in \cref{thm:model_category_Arr(sSet)} corresponds to a Reedy structure on the poset $0<1$. Therefore \cite[Theorem~15.3.4]{Hirschhorn2003} tells us that this model structure is left proper and 
\cite[Theorem~15.7.6]{Hirschhorn2003}
that it is also cellular. Hence functorial localization functors exist in the category of arrows, see \cite[Section~4.3]{Hirschhorn2003}.

\begin{defn}
[{\cite[3.1.4.1]{Hirschhorn2003}}]
Let $w\colon W_0\to W_1$ be a map between pointed spaces. We say that a map between pointed spaces $f\colon X_0\to X_1$ is \emph{$w$-null} if $f$ is fibrant (see Theorem~\ref{thm:model_category_Arr(sSet)}) and the mapping space $\map(w,f)$ is contractible.
A morphism of arrows $\alpha\colon g \to h$ is a \emph{$w$-local equivalence} if $\map(\alpha,f)\colon \map(h,f)\to \map (g,f)$ is a weak homotopy equivalence for all $w$-null maps $f$.
\end{defn}

We are ready for the formal definition of a nullification functor. For each arrow, there is a ``best possible'' approximation by a null arrow. 

\begin{defn}
A \emph{$(w\colon W_0\to W_0)$-nullification functor} is a homotopy idempotent coaugmented functor $\PP_w$ that assigns to any arrow $f\colon X_0\to X_1$ a $w$-null arrow $\PP_wf$ together with a $w$-local equivalence $q\colon f\to \PP_wf$.
\end{defn}

Nullification functors exist as mentioned above, and just as any localization functor, enjoy a universal property, making the informal approximation idea precise.

\begin{prop}[{\cite[3.2.19]{Hirschhorn2003}}]\label{prop:characterization_nulification}
Let $w\colon W_0\to W_1$ and $f\colon X_0\to X_1$ be a pair of maps between pointed spaces, and $q\colon f\to \PP_wf$ the $w$-nullification map of $f$. If $f\colon X_0\to X_1$ is cofibrant, then for every $w$-null object $g$, and every commutative square $\alpha\colon f\to g$, there is a commutative square $\beta \colon \PP_wf\to g$, unique up to homotopy, such that $\beta\circ q=\alpha$. This is, the following diagram commutes:
\[
\begin{tikzcd}
X_0 \arrow[d, "f"] \arrow[r, "q_0"] \arrow[rr, "\alpha_0", bend left] & Z_0 \arrow[d, "\PP_wf"] \arrow[r, "\beta_0"] & Y_0 \arrow[d, "g"] \\
X_1 \arrow[r, "q_1"] \arrow[rr, "\alpha_1", bend right]               & Z_1 \arrow[r, "\beta_1"]                   & Y_1.               
\end{tikzcd}
\]
\end{prop}

We turn now to the special case of interest, namely when $w$ is an identity map.

\begin{prop}
Let $W$ be a pointed space and $f\colon  X_0\to X_1$ be a map between pointed spaces. Then, $f$ induces a homeomorphism of spaces
$f^{\lhd}\colon \map(W, X_0) \longrightarrow \map(\id_W,f)$.
\end{prop}

\begin{proof}
A map of squares is completely determined here by the top arrow. More formally $f^{\lhd}$ sends an $n$-simplex $g\colon W\wedge \Delta[n]_+\to X_0$ to the commutative square:
\[
\begin{tikzcd}
W\wedge \Delta[n]_+\arrow[r,"g"]\arrow[d,"\id_W\wedge \id_{\Delta[n]_+}"']&  X_0\arrow[d,"f"]\\
W\wedge \Delta[n]_+\arrow[r,"f\circ g"]&  X_1.
\end{tikzcd}
\] 
which defines an $n$-simplex of the mapping space in $\Arr(\sSet_*)$.
\end{proof}

This yields immediately a characterization of $\id_W$-null arrows.

\begin{cor}
\label{cor:idnull}
Let $W$ be a pointed space, and $f\colon X_0\to X_1$ be a map between pointed spaces. Then $f$ is $\id_W$-null if and only if $X_0$ is $W$-null. \hfill{$\square$}
\end{cor}

We conclude this section with a description of the $\id_W$-nullification functor. The intuitive idea behind this construction is clear: Up to homotopy, the nullification $\PP_{\id_W}$ is constructed by successively killing all maps from $\id_W$ and here we are thus killing the same maps out of $W$ into the source and the target until the source becomes $W$-null.

\begin{prop}\label{prop:nullification_respect_identity}
Let $W$ be a pointed space, and $f\colon  X_0\to X_1$ be a map between pointed connected spaces. Then $\PP_{\id_W}f\colon  Y_0 \to Y_1$, the $\id_W$-nullification of $f$, is the homotopy pushout of $f$ along the $W$-nullification $X_0 \rightarrow \PP_W X_0 = Y_0$, i.e., the following diagram is a homotopy pushout.
\[
\begin{tikzcd}
X_0 \arrow[d, "f"] \arrow[r] & \PP_W X_0 \arrow[d, "\PP_{\id_W}(f)"] \\
X_1 \arrow[r]                       & Y_1.
\end{tikzcd}
\]
\end{prop}

\begin{proof}
We divide the proof into two steps. First, we prove that $Y_0=\PP_W X_0$. We know by Corollary~\ref{cor:idnull} that $Y_0$ is $W$-null. Let $j\colon X_0\to Z$ be any map to a $W$-null space. We show that it factors through $Y_0$. Indeed, we can extend the nullification map to a commutative diagram of solid arrows:
 \[
\begin{tikzcd}
 X_0 \arrow[r] \arrow[d] \arrow[rr, "j", bend left] & Y_0 \arrow[d] \arrow[r, dashed] & Z \arrow[d] \\
X_1 \arrow[r]                                             & Y_1 \arrow[r]                  & \ast.      \end{tikzcd}
\]
By Proposition~\ref{prop:characterization_nulification}, there exists a dashed factorization through the $\id_W$-nullification of~$f$.
But this is a factorization of $j$ through $Y_0$, showing that $Y_0$ is initial up to homotopy among $W$-null spaces, hence $Y_0\simeq \PP_W X_0$.

To prove that $Y_1\simeq \hocolim(X_1\leftarrow  X_0\to \PP_W X_0)$, we follow a similar argument. We replace $f$ with a cofibration if needed and construct the
strict pushout $P$ of the diagram $X_1 \xleftarrow{f}  X_0 \rightarrow  Y_0$. We prove that the left-hand side square in the diagram below verifies the universal property of the nullification, namely that any morphism $\alpha$ from $f$ to an $\id_W$-null map $h\colon Z_0 \rightarrow Z_1$ factors uniquely through $\PP_W  X_0 \rightarrow P$:
\[
\begin{tikzcd}
X_0 \arrow[r] \arrow[d, hookrightarrow, "f"'] \arrow[rr, "\alpha_0", bend left] & \PP_W X_0 \arrow[d, hookrightarrow, "f'"] \arrow[r, dashed] &  Z_0 \arrow[d, "h"] \\
X_1 \arrow[r] \arrow[rr,"\alpha_1"', bend right] & P \arrow[r, dashed] & Z_1.
\end{tikzcd}
\]
Let us thus assume that $h$ is $\id_W$-null, i.e. $Z_0$ is $W$-local. Therefore the top arrow $\alpha_0$ factors through $\PP_W X_0$ up to homotopy. We may assume that the localization map $X_0 \rightarrow \PP_W X_0$ is a cofibration and that $Z_0$ is fibrant to make it strictly commutative. 
We conclude by the universal property of the pushout that the dashed arrow $P \to Z_1$ also exists and makes the diagram commute. This proves that the morphism $f \rightarrow f'$ induces a bijection on $0$-simplices of mapping spaces $\map(f', h) \rightarrow \map(f, h)$. The same argument applies to higher simplices so that $Y_1\simeq P$.
\end{proof}
\section{Plus-construction}
The integral plus construction for a space $X$, as introduced by Quillen in \cite{Quillen1971}, is a map $q\colon X\to X^+$ such that $H_*(q;\ZZ)$ is an isomorphism and $\pi_1(q)\colon \pi_1(X)\to \pi_1(X^+)$ is an epimorphism whose kernel is the maximal perfect (normal) subgroup $P$ of $\pi_1(X)$. The standard method for building this map can be found in textbooks, such as Hatcher's \cite[p. 373-374]{Hatcher2002}, and we will come back to it in the next section.

\subsection{Plus-constructions and nullification functors}

Bousfield explains in \cite[Theorem~4.4]{MR1481817} that each homotopical localization functor has a best possible approximation by a nullification functor. When the localization functor is localization with respect to ordinary homology with integral coefficients, we obtain a functorial Quillen plus construction as observed by Casacuberta in \cite[Theorem~6.5]{MR1290581}. He also notes that one can take the homotopy cofiber $A$ of a universal homology equivalence so that $\PP_A X \simeq X^+$.
This provides a direct way to generalize Quillen's plus construction to any homology theory. Let $h$ be an arbitrary generalized homology theory. There exists a \emph{universal} $h$-equivalence $f$ such that the localization with respect to all $h$-equivalences coincides with the localization with respect to $f$, an idea that goes back at least to Bousfield's \cite{MR380779}.

\begin{defn}
\label{def:plus}
Let $h$ be a homology theory, $f$ a universal $h$-equivalence and $A$ its homotopy cofiber.
The \emph{$h$-plus construction} of a space $X$ is the $A$-nullification of $X$, so $X^{+h} = \PP_A X$.
\end{defn}

This appears for example in \cite[Remark~6.A.1.1]{Farjoun1996}. Functoriality of the plus construction allows us for example to identify the first algebraic K-theory groups by playing with fibration sequences coming from group extensions related to $\mathrm{GL}(R)$, see \cite[Section~3]{MR1775748}, replacing the ``up to homotopy'' arguments from \cite[Section~5]{MR649409} by strict commutativity.
Let us record a few well-known facts.

\begin{lem}
\label{lem:propertyplusconstruction}
Given a homology theory $h$ and a connected space $X$, the $h$-plus construction $q\colon X\to X^{+h}$ induces an isomorphism in homology, $h_*(q)\colon h_*(X)\to h_*(X^{+h})$ and an epimorphism on the fundamental group, $\pi_1(q)\colon \pi_1(X)\to \pi_1(X^{+h})$.
\end{lem}

\begin{proof}
    Any nullification functor induces an epimorphism on $\pi_1$ by \cite[Proposition~2.9]{MR1257059}, and since we choose here an $h$-acyclic space $A$, then $X \to \PP_A X$ is an $h$-equivalence.
\end{proof}

\begin{rmk}
\label{rem:acyclization}
The homotopy fiber of the $A$-nullification is the $A$-acyclization functor when $A$ is a universal $h$-acyclic space. This has been first studied in \cite{MR315713} for ordinary homology. When $h = HR$ is ordinary homology with coefficients in $R$, we write $\Acy^R X$ for the homotopy fiber of $X \rightarrow X^{+ R}$. 
\end{rmk}

\begin{defn}
\label{def:acyclic}
Let $h$ be a homology theory and $A$ a universal $h$-acyclic space. The \emph{acyclization} functor $\Acy^h$ associates to each space $X$ the homotopy fiber of the nullification map $X \rightarrow X^{+ h}$.
\end{defn}

A group is perfect if it is equal to its commutator subgroup, and this is equivalent in turn with $K(G, 1)^+$ being simply connected. This motivates the following generalization, see \cite[Definition~2.1]{Mislin2001}, where the following notion is called $h1$-perfect.

\begin{defn}\label{def:h-perfect_group}
Let $h$ be a homology theory. A group $G$ is \emph{$h$-perfect} if $K(G,1)^{+h}$ is simply connected. 
\end{defn}

For generalized homology theories, interesting variants have been studied for higher Eilenberg-Mac Lane spaces by Mislin and Peschke.

\subsection{Plus-construction for connective homology theories}
 In this paper, we will only concentrate on connective homology theories. We follow the strategy in \cite{Mislin2001}, in the sense that we use the acyclization functor to understand the plus construction. Then, thanks to Bousfield's work \cite{Bousfield1974}, we reduce to the study to ordinary homology theories with coefficients in a nice abelian group, that he associates to any homology theory.
 
\begin{defn}\label{def:h-characteristic}
Let $h$ be a connective homology theory, $I$ be the set of primes $p$ such that $h_*(\operatorname{pt})$ is uniquely $p$-divisible, and $J$ the set of primes not in $I$. We define the \emph{$h$-characteristic group} 
\[
S_h=\left\{\begin{matrix}
\oplus_{p\in J}\ZZ/p & \text{ if } h_*(\operatorname{pt}) \text{ is torsion,}\\
\ZZ[I^{-1}]& \text{ if } h_*(\operatorname{pt})\text{ is not torsion.}\\
\end{matrix}
\right.
\]
\end{defn}

\begin{exm}\label{exm:h-characteristic}
If $h = k(n)$ is connective Morava K-theory, then $S_{k(n)} = \mathbb Z/p$. We are mostly interested in ordinary homology theories $HR$ for a ring $R$. When $R = \mathbb Q \times \mathbb Z/p$, the characteristic group is $\mathbb Z_{(p)}$, and for $R = \prod \mathbb Z/p$, where the product runs over all primes, the characteristic group is $\mathbb Z$. When $R= \overline{\mathbb F}_p$ is the algebraic closure of the field of $p$ elements, $p$ a prime, then $S_{HR} = \mathbb Z/p$.
\end{exm}

\begin{thm}[{\cite[Theorem 1.1, Theorem 4.5]{Bousfield1974}}]\label{thm:homology_decomposition_ring}
Let $h$ be a connective homology theory and $S_h$ be the associated $h$-characteristic group.
Then $h$ has the same acyclic spaces as $H_*(-;S_h)$.
\end{thm}

For simplicity, we call $S$ the characteristic group $S_{h}$ associated to the connective homology theory $h$ as explained in \cref{def:h-characteristic}. Bousfield's result tells us that the plus constructions associated to $h$ and $H_*(-; S)$ are the same. 
By \cref{thm:homology_decomposition_ring}, it is sufficient to understand the plus-construction associated to ordinary homology with coefficients in a torsion group of the form $\oplus_J \mathbb Z/p$ or a subring $\mathbb Z[I^{-1}]$ of the rationals. Furthermore, in the torsion case, a space is $H_*(-; \oplus_{p\in J} \mathbb Z/p)$-acyclic if and only if it is $H\mathbb Z/p$-acyclic for all $p \in J$. This means we only need to understand the plus-construction in two cases: for subrings of $\mathbb Q$ and for $\mathbb Z/p$.

In the case when $S = \ZZ$, see \cite{Berrick1999}, or $\ZZ/p$, for $p$ prime, see \cite{Casacuberta1999}, there are well-known and beautiful models of universal acyclic spaces associated with word radicals, \cite{Casacuberta2005}. For $S=\ZZ[I^{-1}]$, general methods guarantee the existence of a universal acyclic space, but we do not know of a small and concrete model in the spirit of the previous examples.

\subsection{Perfect and strongly perfect groups}
Given a connective homology theory $h$ and a group $G$, \cref{def:h-perfect_group} seems to be unpractical, in general, to determine if $G$ is $h$-perfect. In this section, we characterize $h$-perfect groups in purely algebraic terms, by using the characteristic group $S_h$ introduced in \cref{def:h-characteristic}. Moreover, when $h$ is an ordinary homology theory with coefficients in a ring $R$, we characterize it in terms of~$R$. Recall the classical definition of $R$-perfectness, and the more recent and stronger version from~\cite{Broto2021}.

\begin{defn}
\label{exm:stronglyperfect}
Let $R$ be a commutative ring $R$ with unit and consider $HR$, ordinary homology with coefficients in $R$. A group $G$ is \emph{$R$-perfect} if $H_1(G;R)=0$,  equivalently, if $H_1(G;\ZZ)\otimes_\ZZ R=0$. We say that $G$ is \emph{strongly $R$-perfect} if $G$ is $R$-perfect and $\operatorname{Tor}(H_1(G;\ZZ),R)=0$.
\end{defn}

Note that in general, for a ring $R$ of characteristic $0$, being $HR$-perfect in the sense of \cref{def:h-perfect_group} is not equivalent to be $R$-perfect. We will come back to this at the end of this section, see \cref{cor:hperfectisSperfect}.

\begin{rmk}\label{rem:abuse}
The above definition also makes sense when $R$ is not a ring, but a mere abelian group. We will use this slightly abusive terminology in the sequel and mention $S$-perfect groups for an abelian group $S$ which is not a ring in general. An important example is that of an infinite direct sum of cyclic groups of prime order $\ZZ/p$. In this case $S$-perfect means $\ZZ/p$-perfect for all primes $p$ appearing in the direct sum.
\end{rmk}

To give a characterization of $h$-perfectness we need the following algebraic result.

\begin{lem}\label{lem:R_A_same_invertibles}
Let $R$ be a ring of characteristic $0$ and $S$ be the associated $HR$-characteristic group. An integer $k\in \ZZ$ is invertible in $R$ if and only if it is so in $S$.
\end{lem}

\begin{proof}
Since $R$ is of characteristic $0$, then $S$ is $\ZZ[I^{-1}]$ for $I$ the set of primes $p$ such that $R$ is uniquely $p$-divisible. Let $k\in \ZZ$ be invertible in $R$, and let $p_1^{\alpha_1}\dots p_n^{\alpha_n}$ be the decomposition in prime numbers of $k$. Since $k$ is invertible, then $p_i$ is invertible too.
That implies that $R$ is uniquely $p_i$-divisible, so $p_i\in I$, and therefore, $k=p_1^{\alpha_1}\dots p_n^{\alpha_n}$ is invertible in $S$. Using a similar argument, one can prove that if $k\in \ZZ$ is invertible in $S$, then $k$ is invertible in $R$ too.
\end{proof}

\begin{prop}\label{prop:strongly_iff_perfect}
Let $R$ be a commutative ring of characteristic $0$ and $S$ be the associated $HR$-characteristic group. For a group $G$, the following are equivalent:
\begin{enumerate}
\item $G$ is strongly $R$-perfect
\item $G$ is strongly $S$-perfect.
\item $G$ is $S$-perfect.
\end{enumerate}
\end{prop}

\begin{proof}
Since $S$ is torsion free by construction (it is a subring of the rational numbers $\mathbb Q$), the equivalence between (2) and (3) is obvious.

We are thus left to prove that (1) and (2) are equivalent. Broto, Levi, and Oliver provide in \cite[Lemma A.3]{Broto2021} a characterization of strongly perfect groups over a ring of characteristic zero. Each element of the group $G$ must be annihilated by an integer that is invertible in the ring. We conclude then by Lemma~\ref{lem:R_A_same_invertibles} since $R$ and $S$ contain the same invertible integers.
\end{proof}
\begin{exm}
We consider the ring $R=\QQ\times \ZZ/p$ for some prime number $p$, as in \cite[Example~A.6]{Broto2021} and the associated characteristic group $S=\ZZ_{(p)}$. The Pr\" ufer group $\ZZ_{p^{\infty}}$ is $R$-perfect but not strongly $R$-perfect, and therefore not $S$-perfect. In particular, the plus construction $(B\ZZ_{p^{\infty}})^{+R}$ cannot kill the fundamental group. In fact this classifying space is local, $(B\ZZ_{p^{\infty}})^{+R} \simeq B\ZZ_{p^{\infty}}$.
\end{exm}

In Lemma~\ref{lem:propertyplusconstruction} we saw that the plus construction induces an epimorphism on fundamental groups, we can do better and say something about the kernel.

\begin{prop}
\label{prop:perfectkernel}
Let $h$ be a connective homology theory, and $S$ be the associated $h$-characteristic group. For any connected space $X$, the kernel of $\pi_1(X) \rightarrow \pi_1 (X^{+ h})$ is $S$-perfect. Moreover, if $h$ is ordinary homology theory with coefficients in a ring $R$ with unit, the kernel of $\pi_1(X) \rightarrow \pi_1 (X^{+ h})$ is:
\begin{itemize}
    \item strongly $R$-perfect if $\chara(R)=0$; or
    \item $R$-perfect if $\chara(R)\not=0$.
\end{itemize}
\end{prop}

\begin{proof}
    We know from Bousfield's Theorem~\ref{thm:homology_decomposition_ring} that $X^{+ h} \simeq X^{+ S}$. Therefore, using \cref{prop:strongly_iff_perfect}, it is enough to prove that the kernel of $\pi_1(X)\rightarrow \pi_1(X^{+S})$ is $S$-perfect.
    From the long exact sequence in homotopy associated to the fibration sequence $\Acy^S X \to X \to X^{+ S}$, we see that the kernel of the induced map $\pi_1(X) \rightarrow \pi_1 (X^{+ S})$ is the image of $\pi_1 (\Acy^S X)$. But $H_1(\Acy^S X; S) \cong H_1(\Acy^S X; \ZZ) \otimes S \cong \left(\pi_1 (\Acy^S X)\right)_{ab} \otimes S$ is zero since $\Acy^S X$ is $HS$-acyclic. Hence $\pi_1 (\Acy^S X)$ is $S$-perfect and, by right exactness of abelianization and the tensor product, so is any quotient.

    When $h$ is ordinary homology theory with coefficients in a ring $R$, the result holds by \cref{prop:strongly_iff_perfect}.
\end{proof}

We extract as a direct corollary the following characterization of $h$-perfect groups.
\begin{cor}
\label{cor:hperfectisSperfect}
    Let $h$ be a connective homology theory, and $S$ be the associated $h$-characteristic group. A group $G$ is $h$-perfect if and only if $G$ is $S$-perfect.
\end{cor}
\begin{proof}
    This result holds directly by applying \cref{prop:perfectkernel} to \cref{def:h-perfect_group}
\end{proof}

\section{Relative plus constructions}
Relative versions of Quillen's plus construction were already present in \cite{Quillen1971}. This plus construction for a space $X$ relative to a perfect subgroup $H\unlhd \pi_1(X)$ is a map \mbox{$q\colon X\to X^+_H$} such that $H_*(q;\ZZ)$ is an isomorphism and $\pi_1(q)\colon \pi_1(X)\to \pi_1(X^+_H)$ is an epimorphism with kernel $\ker(\pi_1(q))=H$. A possible, and brief, description is the following. We consider a covering space $p\colon \widetilde X\to X$ that corresponds to $H$, this is, $\pi_1(\widetilde X)\cong H$. Then, consider $\widetilde X^+$ the space obtained from $\widetilde X$ by attaching $2$- and $3$-dimensional cells to \emph{kill} the fundamental group of $\widetilde X$ without changing its homology. We can see $\widetilde X$ as a subspace of $\widetilde X^+$. Then, we define $X^+_H$ to be the homotopy pushout, see \cite[Section~4.2]{Hatcher2002}:
\[
\begin{tikzcd}
\widetilde X \arrow[r, hook] \arrow[d] & \widetilde X^+ \arrow[d] \\
X \arrow[r, "q"]                   & X^+_H.            
\end{tikzcd}
\]

Recently, Broto, Levi and Oliver \cite{Broto2021} introduced a \emph{relative} plus construction for homology with coefficients in a ring $R$. The relativity refers to a non-maximal (strongly) $R$-perfect subgroup.

\begin{defn}
[{\cite{Broto2021}}]\label{def:BLO+const}
    Fix a commutative ring $R$ with unit, a pointed space $X$, and a normal subgroup $H\unlhd \pi_1(X)$. An \emph{$R$-plus construction for $(X,H)$}, if it exists, is a space $X^{+R}_H$ together with a map $q\colon X\to X^{+R}_H$, such that $\pi_1(q)$ is an epimorphism with kernel $H$, and $H_*(q;N)$ is an isomorphism for each $R[\pi_1(X)/H]$-module $N$. 
\end{defn}

\begin{thm}[{\cite[Proposition A.5]{Broto2021}}]
Let $R$ be a commutative ring with unit, let $X$ be a connected CW complex, and let $H\unlhd \pi_1(X)$ be a normal subgroup. Then there is an $R$-plus construction for $(X,H)$ (see~\ref{def:BLO+const}) if and only if either,
\begin{itemize}
    \item $\chara(R)\not =0$ and $H$ is $R$-perfect, or
    \item $\chara(R)=0$ and $H$ is strongly $R$-perfect.
\end{itemize}
\end{thm}

They prove the ``if'' implication using Quillen's technique of attaching $2$- and $3$-dimen\-sional cells to kill the desired subgroup. With this definition, the plus construction is not unique, and, moreover, in contrast with the classical plus construction, it does not enjoy any universal property. Here is a very simple example illustrating the kind of problems that one can encounter.

\begin{exm}
\label{exm:BLO-not-well-defined}
    Let $R = \ZZ/3$ and $\RR P^2$ be the real projective plane. Notice that the suspension $\Sigma \RR P^2$ is simply-connected and $2$-torsion, hence $H\ZZ/3$-acyclic, so in particular $\widetilde H_*(\Sigma\RR P^2;\ZZ/3)=0$. 
    
    On the one hand, the identity $\id\colon \Sigma \RR P^2\to \Sigma\RR P^2$ is the plus construction given by Broto, Levi, and Oliver's proof. But, on the other hand, since $\Sigma\RR P^2$ is $H\ZZ/3$ acyclic, the nullification with respect to the universal $H\ZZ/3$-acyclic space is contractible, so \mbox{$\Sigma \RR P ^2\to \ast$} is also a plus construction in the sense of Definition~\ref{def:BLO+const}.
\end{exm}

In order to fix this issue we work in the arrow category. Just like classical treatments do, see Hatcher's presentation in \cite[p. 373-374]{Hatcher2002}, which we already mentioned in the introduction of the present section, and also \cite{Broto2021}, we use the covering space associated to the given subgroup $H$. This subgroup must be normal as our aim is to kill exactly this subgroup in the process and it must be perfect with respect to the homology theory we are using since we wish to obtain a homology equivalence.

\begin{defn}\label{def:universal_plus_construction}
    Let $h$ be a connective homology theory and $S$ be the associated characteristic group. Let $X$ be a connected space and $H\unlhd \pi_1(X)$ a normal $S$-perfect subgroup of $\pi_1(X)$. Let $p\colon \widetilde X\to X$ be a covering space corresponding to $H$. The \emph{relative plus construction} of $X$ relative to $H$, that we will denote by $p^{+h}$, is the $\id_A$-nullification of $p$, where $A$ is a universal $HS$-acyclic space.
\end{defn}
\begin{rmk}\label{rem:universal_plus_construction}
Proposition~\ref{prop:nullification_respect_identity} tells us that this nullification in the arrow category can be obtained as a homotopy pushout:
    \[
    \begin{tikzcd}
    \widetilde X \arrow[r, hook] \arrow[d] & \widetilde X^{+h} \arrow[d] \\
    X \arrow[r, "q"]                   & X^{+h}_H.               
    \end{tikzcd}
    \]
We also remark that \cref{def:universal_plus_construction} is also valid for a homology theory $h$ and any $h$-perfect group~$H$.
\end{rmk} 

We have to relate this construction with the way Broto, Levi, and Oliver do it. This means that we have to verify that the map $q\colon X \to X^{+h}_H$ enjoys all properties of a relative plus construction as defined in \cref{def:BLO+const}. We start with the more elementary parts and keep the local coefficients for the end.

\begin{thm}
\label{thm:elementaryproperties}
    Let $h$ be a connective homology theory,  and $S$ be the associated $h$-charac\-teristic group. Let $X$ be a connected space and $H\unlhd\pi_1(X)$ a normal $S$-perfect subgroup of~$X$. The relative plus construction $q\colon X\to X^{+h}_H$ induces an isomorphism in homology $h_*(q)\colon h_*(X)\to h_*(X^{+h}_H)$ and an epimorphism on the fundamental group $\pi_1(q)\colon \pi_1(X)\to \pi_1(X^{+h}_H)$ with kernel $\ker(\pi_1(q))=H$.
\end{thm}

\begin{proof}
Let $p\colon \widetilde X\to X$ be a covering map corresponding to $H$. We know from Remark~\ref{rem:universal_plus_construction} that a homotopy pushout diagram describes the relative plus construction.
First, note that by Seifert-van Kampen, see, for example, the very general version \cite[Corollary~5.1]{DrorFarjoun2004}, $\pi_1(q)$ is an epimorphism with kernel $\ker\pi_1(q)=H$. The (absolute) plus construction for $HS$ kills the maximal $S$-perfect subgroup, by Mislin and Peschke's \cite[Lemma~2.16(iv)]{Mislin2001}. This is precisely $H$ since $h$-perfect groups coincide with $S$-perfect ones by \cref{cor:hperfectisSperfect}, see also \cite[Section~5]{Mislin2001}.

To conclude, we must prove that $h_*(q)$ is an isomorphism. This can be done by a direct computation from the long exact sequence in homology associated to the homotopy pushout, or just observe that the top horizontal homotopy cofiber is $HS$-acyclic since the (absolute) plus construction is a homology equivalence, therefore so is the weakly homotopy equivalent bottom horizontal homotopy cofiber.
\end{proof}

We have obtained at this point a homology equivalence $X \to X^{+ h}_H$ killing a chosen $S$-perfect normal subgroup $H \unlhd \pi_1(X)$, but, if $h$ is a homology theory with coefficients in a ring $R$, in order to deserve the name of a relative plus construction it should also be a homology equivalence for local coefficients.

\begin{thm}\label{thm:universa_plus_construction_local_coeff}
Let $R$ be a commutative ring with unit and $S$ be the associated $HR$-characteristic group. Let $X$ be a connected space and $H\unlhd\pi_1(X)$ an $S$-perfect subgroup of $\pi_1(X)$. The relative plus construction $q\colon X\to X^{+R}_H$ induces an isomorphism $H_*(q;N)\colon H_*(X;N)\to H_*(X^{+R}_H;N)$ in homology with coefficients in every $R[\pi_1(X)/H]$-module~$N$.
\end{thm}

To prove this theorem, we will use the following version of a lemma by Broto, Levi, and Oliver.

\begin{lem}[{\cite[Lemma A.2]{Broto2021}}]\label{lem:auxiliar_local_coef}
Let $R$ be a commutative ring with unit and \mbox{$f\colon X\to Y$} be a map between connected spaces inducing an epimorphism $\pi_1(f)\colon G \twoheadrightarrow \pi$ with kernel~$H$. 
Let $\widetilde X$ and $\widetilde Y$ be the covering spaces of $X$ and $Y$ with fundamental groups 
$H$ and $1$ respectively.
Assume that a covering map $\widetilde f \colon \widetilde X\to \widetilde Y$ is an $R$-homology equivalence. Then $H_*(f;N)$ is an isomorphism for each $R\pi$-module $N$.
\end{lem}

\begin{proof}[Proof of~\ref{thm:universa_plus_construction_local_coeff}]
Let $p\colon \widetilde X\to X$ be a covering map corresponding to $H$, and let $p^{+R}\colon \widetilde X^{+R}\to X^{+R}_H$ be the plus construction of $p$ where $A$ is a universal $HS$-acyclic space. Let $E\to X^{+R}_H$ be the universal cover of $X^{+R}_H$. As $\widetilde X^{+R}$ is simply connected, the 
map $p^{+R}$ factors through $E$.
We can then enlarge the nullification diagram
\[
\begin{tikzcd}
\widetilde X \arrow[r] \arrow[d, "p"'] & \widetilde X^{+R} \arrow[r] \arrow[d, "p^{+R}"]        & E \arrow[d] \\
X \arrow[r, "q"]                  & X^{+R}_H \arrow[r, Rightarrow, no head] & X^{+R}_H
\end{tikzcd}
\]
where the left-hand side square is a homotopy pushout square, see Remark~\ref{rem:universal_plus_construction}. The homotopy fiber of the top map (the plus construction of $\widetilde X$) is $\Acy^S \widetilde X$, the acyclization functor for $HS$, see Remark~\ref{rem:acyclization}. We deduce then from \cite[Theorem~3.4(I)]{Chacholski1997} that the homotopy fiber $\operatorname{Fib}(q)$ is also $HS$-acyclic. Consider next the outside square:
\begin{equation*}\label{equ:local_coef_pullback}
\begin{tikzcd}
\widetilde X \arrow[r,"\widetilde q"] \arrow[d] & E \arrow[d] \\
X \arrow[r,"q"]                  & X^{+R}_H.
\end{tikzcd}
\end{equation*}
Both vertical maps are covering maps corresponding to free actions of the same group on the covering spaces. This outside square is therefore a homotopy pullback square (see for example \cite[Theorem~14.1.10]{MR2456045}) and horizontal homotopy fibers are thus weakly homotopy equivalent. All together we have proved that the homotopy fiber $\operatorname{Fib}(\widetilde q)$ is $HS$-acyclic, in particular, $\widetilde q$ is an $HS$-equivalence.
We conclude by Lemma~\ref{lem:auxiliar_local_coef} that $q$ induces an isomorphism in homology with all desired local coefficients. 
\end{proof}
This shows that our Definition~\ref{def:universal_plus_construction} of a relative plus construction for ordinary homology with coefficients in any ring verifies all desired properties from Broto-Levi-Oliver's Definition~\ref{def:BLO+const}.

\section{The universal property of the relative plus construction}
\label{sec:universal}
It is good to know that our nullification procedure in the category of arrows produces a relative plus construction in the sense of \cref{def:BLO+const}, but why is it better? Because it is a localization, it enjoys two universal properties. The map $p^{+h}\colon\widetilde X^{+h} \rightarrow X^{+h}_H$ is the initial map under the covering map $p\colon\widetilde X \rightarrow X$ among all $\id_A$-null maps where $A$ is a universal $h$-acyclic space. This is not so enlightening, let us thus focus on the second universal property, namely that of being terminal among all $\id_A$-local equivalences. 

\begin{lem}
\label{lem:Ganea step}
Let $h$ be a connective homology theory, $Z\rightarrow X$ be a map between connected spaces, and $X\rightarrow Y$ be an $h$-acyclic fibration between connected spaces with $n$-connected fiber $F$. Assume that the fiber inclusion $F\rightarrow X$ factors through $Z$. Then, the homotopy fiber $F_1$ of the induced map $X/F\to Y$ is $h$-acyclic and $(n+1)$-connected, and the fiber inclusion $F_1\rightarrow X/F$ factors through $Z/F$.
\end{lem}

\begin{proof}
The setup described in the lemma can be summed up in the following diagram:
\[
\begin{tikzcd}
    F \ar[d, equal] \ar[r] & Z \ar[d] \ar[r] & Z/F \ar[d]& \\
    F \ar[r] & X \ar[r] & X/F \ar[r] & Y.
\end{tikzcd}
\]
We compute the homotopy fiber $F_1$ as the join $F * \Omega Y$ by Ganea's result \cite[Theorem~1.1]{MR179791}. Since this is weakly homotopy equivalent to $F \wedge \Sigma\Omega Y$ and $Y$ is connected, $F_1$ is again $h$-acyclic, and the connectivity is at least that of $F$ plus one. 

To prove the second statement, we define the map $F_1 \rightarrow Z/F$ as the homotopy pushout of a diagram of maps:
\[
\begin{tikzcd}
    \Omega Y \ar[d] & \Omega Y \times F \ar[l] \ar[r] \ar[d] & F \ar[d] \\
    * & F \ar[l] \ar[r] & Z.  
\end{tikzcd}
\]
The bottom row comes with a natural transformation of pushout diagrams to $* \leftarrow F \rightarrow X$ by assumption on the fiber inclusion for $F$. The composition of homotopy pushouts $F * \Omega Y \rightarrow Z/F \rightarrow X/F$ is the desired factorization.
\end{proof}

The previous lemma is the induction step in our proof of the universal property for a relative plus construction. The property we are looking for should be stated in terms of the relative plus construction $X \rightarrow X^{+R}_H$, but the universal property we have is given in the category of arrows. This explains why we have to work a little.

\begin{thm}\label{thm:universal property}
Let $h$ be a connective homology theory and $S$ be the associated $h$-character\-istic group. Let $X$ be a connected space and $H\unlhd \pi_1(X)$ an $S$-perfect subgroup of $\pi_1(X)$. The relative plus construction $q\colon X\to X^{+h}_H$ is terminal among all $h$-acyclic maps that kill a subgroup of $H$ on $\pi_1 (X)$.
\end{thm}

\begin{proof}
Let $F_0$ be the fiber of any $h$-acyclic fibration $f_0\colon X \rightarrow Y$. Without loss of generality, we can assume that the subgroup $K$ that $f_0$ kills in $\pi_1(X)$ is $H$ (if not, we obtain a map from $Y$ to $Y^{+h}_{H/K}$ and work with the latter). We apply \cref{lem:Ganea step} to the $h$-acyclic fibration $f_0$ and the covering map $p\colon \widetilde X \rightarrow X$ corresponding to $H$. Therefore we know that $f_0$ factors through a new acyclic fibration $f_1\colon X_1=X/F \rightarrow Y$ and the fiber inclusion factors through $F_1 \rightarrow \widetilde X/F= \widetilde X_1$.

We then iterate this procedure and obtain a sequence of maps $X \rightarrow X_{n-1} \rightarrow X_{n} \xrightarrow{f_{n}} Y$ for $n \geq 1$. By \cref{lem:Ganea step} the homotopy fiber $F_n$ of $f_n$ is $h$-acyclic, in particular connected. The connectivity of the $F_n$'s increases then strictly at each step, so that the homotopy colimit $X_\infty$ of the $X_n$'s is weakly homotopy equivalent to $Y$.

We are now ready to establish the universal property. For this, we have to construct a map $Y \rightarrow X^{+h}_H$, and the previous discussion shows that equivalently we can construct compatible maps out of the $X_n$'s. The induction starts with the map $X=X_0 \rightarrow X^{+h}_H$ which is given as the level $1$ part of a map of arrows from $\widetilde X \rightarrow X$ to $\widetilde X^{+h} \rightarrow X^{+h}_H$. The inductive step is best seen in a diagram:
\begin{equation}\label{equ:prf_universal_property}
\begin{tikzcd}
    F_{n-1} \ar[d, equal] \ar[r] & \widetilde X_{n-1} \ar[d, "p_{n-1}"] \ar[r] & \widetilde X_{n-1}/F_{n-1} \ar[d, "p_n"] \ar[rr, dashed] && \widetilde X^{+h} \ar[d, "p^{+h}"]\\
    F_{n-1} \ar[r] & X_{n-1} \ar[r] & X_{n-1}/F_{n-1} \ar[r] \arrow[rr, dashed, bend right]& Y \ar[r, dotted] & X^{+h}_H.
\end{tikzcd}
\end{equation}
The top dashed arrow exists because $F_{n-1}$ is $h$-acyclic, but to get a compatible pair of dashed arrows, we work in the category of arrows and prove that the relative plus construction $p^{+h}$ is local with respect to the map of arrows $p_{n-1} \rightarrow p_{n}$. Let us recall that $\map (p_{n}, p^{+h})$ and its analog for $p_{n-1}$ are obtained as the homotopy pullbacks of the top and bottom row respectively in
\[
\begin{tikzcd}
    \map(\widetilde X_n, \widetilde X^{+h}) \ar[r] \ar[d] & \map(\widetilde X_n, X^{+h}_H) \ar[d] & \map(X_n, X^{+h}_H) \ar[d] \ar[l] \\
    \map(\widetilde X_{n-1}, \widetilde X^{+h}) \ar[r] & \map(\widetilde X_{n-1}, X^{+h}_H) & \map(X_{n-1}, X^{+h}_H). \ar[l]
\end{tikzcd}
\]
The left hand side vertical map is a weak homotopy equivalence as observed above since $\widetilde X_{n-1} \rightarrow \widetilde X_{n}$ is an $h$-acyclic map and $\widetilde X^{+h}$ is $A$-null. The right hand side square is a homotopy pullback, it has been obtained by mapping a homotopy pushout square into $X^{+h}_H$. Therefore, the induced map $\map (p_{n}, p^{+h}) \rightarrow \map (p_{n-1}, p^{+h})$ is a weak homotopy equivalence.

We conclude then that the pair of dashed arrows in Diagram~(\ref{equ:prf_universal_property}) exist indeed, for any $n$, providing a map $X_\infty \rightarrow X^{+h}_H$ on the homotopy colimit, which is weakly homotopy equivalent to $Y$. This yields the dotted map we need.
\end{proof}

\nocite{*}
\bibliographystyle{alpha}

\end{document}